\documentclass[preprint,11pt]{elsarticle}

\usepackage{amsmath}
\usepackage{amssymb}
\usepackage{amsfonts}
\usepackage{fullpage}

\newtheorem{theorem}{Theorem}

\newtheorem{corollary}[theorem]{Corollary}

\newtheorem{example}[theorem]{Example}

\newtheorem{lemma}[theorem]{Lemma}

\newtheorem{problem}[theorem]{Problem}

\newtheorem{remark}[theorem]{Remark}

\newenvironment{proof}[1][Proof]{\noindent\textbf{#1.} }{\ \rule{0.5em}{0.5em}}
\input{tcilatex}

\journal{Journal of Mathematical Analysis and Applications}

\begin{document}
\begin{frontmatter}

\title{Semicycles and correlated asymptotics of oscillatory solutions to
second-order delay differential equations}

\author[label1]{Elena Braverman}
\author[label2]{Alexander Domoshnitsky}
\author[label2]{John Ioannis Stavroulakis}

\address[label1]{Dept. of Math. and Stats., University of
Calgary,2500 University Drive N.W., Calgary, AB  T2N 1N4, Canada}
\address[label2]{Department of Mathematics, Ariel University, Ariel 4076414, Israel}


\begin{abstract}
We obtain several new comparison results on the distance between zeros and
local extrema of solutions for the second order delay differential equation 
\begin{equation*}
x^{\prime \prime }(t)+p(t)x(t-\tau (t))=0,\text{\ }t\geq s\text{ }\ 
\end{equation*}%
where $\tau :\mathbb{R}\rightarrow \lbrack 0,+\infty )$, $p:\mathbb{R}%
\rightarrow \mathbb{R}$ are Lebesgue measurable and uniformly essentially
bounded, including the case of a sign-changing coefficient. We are thus able
to calculate upper bounds on the semicycle length, which guarantee that an
oscillatory solution is bounded or even tends to zero. Using the estimates
of the distance between zeros and extrema, we investigate the classification
of solutions in the case $p(t)\leq 0,t\in \mathbb{R}.$
\end{abstract}


\end{frontmatter}

\section{Introduction}

\qquad Consider the second order nonautonomous delay differential equation 
\begin{equation}
x^{\prime \prime }(t)+p(t)x(t-\tau (t))=0,  \label{2.41}
\end{equation}%
in the case where $\tau :\mathbb{R}\rightarrow \lbrack 0,+\infty )$, $p:%
\mathbb{R}\rightarrow \mathbb{R}$ are Lebesgue measurable and uniformly
essentially bounded. By a solution of \eqref {2.41}, we understand a
function $x:[s-\tau _{m},+\infty )$, where $\tau _{m}:=\limfunc{esssup}%
\{\tau (t):t\geq $ $s\},$ which is absolutely continuous, together with its
first derivative, such that \eqref {2.41} holds a.e. on $[s,+\infty )$. The
restriction of $x$ to $[s-\tau _{m},s]$ is called the initial data. In the
case $\tau _{m}=0$, we consider the value of the first derivative at $s$ to
be included in the initial data. We call a real function $x:A\rightarrow 
\mathbb{R}
,A\subset 
\mathbb{R}
,$ oscillatory if it has arbitrarily large zeros. An interval between
adjacent zeros of an oscillatory solution is referred to as a\emph{\
semicycle}.

Both the delayed and the infinite dimensional character of \eqref {2.41} are
essential. The space of initial functions engenders an infinite dimensional
solution space, and even infinite dimensional ODEs fail to capture its
complexity. In fact, to rewrite \eqref {2.41} in a more general context, one
would need to consider difference equations in the phase space (\cite[%
Chapters 3, 4]{Hale}), retaining both the delay and the infinite solution
space. In this paper we focus our attention on the family of oscillatory
solutions, which is often the entire solution space, containing both bounded
and unbounded solutions. The classification of oscillatory solutions based
on the length of their semicycles is at the core of several classic
asymptotic results on \eqref {2.41}, obtained by two major veins of
research, of Myshkis and Azbelev. We innovate by considering the semicycles
as a continuum parametrized by the delay, refining and generalizing the
standard two-fold distinction between \textquotedblright
slow\textquotedblright\ and \textquotedblright rapid\textquotedblright\
oscillation. Let us describe in broad strokes these known results.

Myshkis \cite{Myshkis51} studied \eqref {2.41} with continuous parameters
and nonnegative coefficient $p:%
\mathbb{R}
\rightarrow \lbrack 0,+\infty )$, using bounds of the semicycle length. In 
\cite{Myshkis51}, a \textquotedblright large\textquotedblright\ semicycle
with adjacent zeros $a,b$ where $a<b,$ is defined by 
\begin{equation*}
b>\sup \{t:t-\tau (t)<a\},
\end{equation*}%
otherwise, it is a \textquotedblright small\textquotedblright\ semicycle.
Oscillatory solutions only possessing large semicycles are referred to as
slowly oscillating, whereas solutions with small semicycles are rapidly
oscillating. \ Myshkis proved several sufficient conditions for solutions to
be slowly oscillating. As for rapidly oscillating solutions, he showed that
when 
\begin{equation}
\tau _{m}\sqrt{\sup_{t\in 
\mathbb{R}
}p(t)}\leq 2\sqrt{2},  \label{10.27}
\end{equation}%
rapidly oscillating solutions are bounded, and if, further, the strict
inequality holds in \eqref {10.27} they tend to zero at infinity. He also
showed that solutions of the equation 
\begin{equation}
x^{\prime \prime }(t)+x(t-c)=0  \label{02.02}
\end{equation}%
possessing a large semicycle must be unbounded when $c\in (0,\pi ]$ (\cite%
{Myshkis51}, \cite[Chapter 16, p. 303]{Myshkis}), while Hale \cite[p. 135]%
{Hale} showed that \eqref {02.02} has unbounded solutions when $c>\pi $.

In contrast, Azbelev \cite{Azb1971} restricted attention to solutions of %
\eqref {2.41}, with nonnegative coefficient $p,$ corresponding to the zero
initial function%
\begin{equation}
x(t)\equiv 0,t\in \lbrack s-\tau _{m},s)  \label{3.57}
\end{equation}%
(allowing for discontinuity of $x$ or $x^{\prime }$ at $s$), which form a
two-dimensional fundamental system, generated by the solutions $z,y,$
corresponding to $z(s)=1,z^{\prime }(s)=0,$ and $y^{\prime }(s)=1,y(s)=0$. A
similar method was employed in \cite{Norkin}. The finite dimension of this
restricted solution space significantly facilitates the study of asymptotics
and enables one to gauge the behavior of the solutions by estimating the
Wronskian of the fundamental system%
\begin{equation*}
W(t):=\left\vert 
\begin{array}{cc}
z(t) & y(t) \\ 
z^{\prime }(t) & y^{\prime }(t)%
\end{array}%
\right\vert ,t\geq s,
\end{equation*}%
introduced by Domoshnitsky \cite{Dom1983}. Crucial in this vein of research 
\cite{BDK, Dom1983,Dom2001,Paatashvili} is slow oscillation of solutions of %
\eqref{2.41}, with the zero initial function \eqref {3.57}, in the sense of
the $h-$condition (on which see \cite{Azb1971}). \bigskip If the function $%
t\mapsto (t-\tau (t))$ is monotone and the coefficient $p$ nonnegative, all
oscillatory solutions of \eqref {02.02}, satisfying initial condition %
\eqref{3.57}, slowly oscillate. Notably, when furthermore the coefficient $p$
is nondecreasing and positive bounded, the existence of unbounded solutions
to \eqref {02.02}, with initial condition \eqref {3.57}, is equivalent to
the divergence of $\int_{0}^{\infty }\tau (t)dt$ (see \cite{Dom2001, iz}).

Unbounded solutions proven to exist within this framework are generally
dominant slowly oscillating solutions, as oscillation of \eqref {2.41} with
bounded parameters and positive coefficient is equivalent to oscillation of
the ODE $x^{\prime \prime }(t)+p(t)x(t)=0$, which is known to oscillate
under quite weak conditions (\cite{brands, dzurina, hille, kwong}). For an
in-depth study of several such parallels in asymptotic behavior and
classification of solutions between \eqref {2.41} and the corresponding ODE,
the reader is referred to Do\v{s}l\'{a} and Marini \cite{dosla}.
Furthermore, inequality \eqref{10.27}, which implies that rapidly
oscillating solutions are bounded, also implies \textit{a fortiori} slow
oscillation of the unbounded dominant solutions. However, while the constant 
$2\sqrt{2}$ is sharp in \eqref{10.27} by construction regarding the
boundedness of rapidly oscillating solutions, it can be improved to $\frac{%
\pi }{2}+\sqrt{2}$ regarding slow oscillation (see \cite[p. 18-23]{Myshkis51}
and independent results by Eliason \cite{eliason}).

In summary, both the research of Myshkis \cite{Myshkis51}, \cite{Myshkis},
and that of Azbelev \cite{Azb1971} and Domoshnitsky \cite{Dom2001}, beg the
question: is slow oscillation equivalent to unboundedness? This was answered
in the negative by Myshkis \cite[Paragraph 6]{Myshkis51}, constructing
examples of bounded, slowly oscillating, solutions, under $\tau _{m}\in (0,%
\frac{\pi }{2})$. Surprisingly, despite many well-known results linking slow
and rapid oscillation to unboundedness and boundedness respectively, the
distinction between rapidly and slowly oscillating solutions is insufficient
to determine their asymptotics. While Myshkis did consider the entire
solution space, the classification of slow and rapid oscillation was not a
fine enough sieve to describe (un)boundedness in terms of semicycle length.

Indeed, instead of a two-fold solution space (slowly and rapidly oscillating
solutions), one expects an infinite dimensional solution space, containing
oscillatory solutions that tend to zero with arbitrarily small semicycles.
In the case where there exists a spectral decomposition (see \cite[Chapter
12, p. 417]{bellman}, \cite[Chapter 8]{Hale}), either there exists a
sequence of eigenvalues that tend to infinity in the complex norm, and the
corresponding semicycles (half-periods) of the solutions tend to zero, or
there exist \textit{small}, superexponential, solutions which are expected
to be oscillatory. On the (non)existence of small solutions and their
oscillation, see e.g. \cite{cooke, garab}. More generally, the pathologies
exhibited by infinite dimensional subspaces of the phase space should
prevent the families of positive and unbounded solutions from\ being
infinite. For example, closed infinite subspaces of $\mathcal{C}[0,1]$, must
themselves contain infinite dimensional subspaces comprised of functions
with infinitely many zeros, as a consequence of well-known spaceability
results \cite[Corollary 3.7, Corollary 3.8]{enflo}, \cite{levine}. The
solution space being infinite is the expected behavior, one requiring 
\textit{ad hoc }conditions such as the vanishing of the coefficient or the
delay to ensure finite solution spaces (see sufficient conditions for an
infinite solution space in \cite{halepaper, lillo}, \cite[Chapter 2]{Hale}).
Thus, the solution space will typically consist of finite dimensional
unstable and center manifolds, along with an infinite dimensional stable
manifold of oscillatory solutions that have arbitrarily small semicycles.

\bigskip Moreover, the distance between zeros is intimately linked to the
rate of growth/decay of oscillatory solutions, and this relation has been
the subject of many investigations, e.g. \cite{cao1, cao2, garab, krisztin,
kriwalther}, starting with the Morse decompositions \cite{mallet}. Namely, a
bound on the oscillation frequency implies a bound on the rate of decay. Our
point of interest in this paper is existence and estimates of a threshold
semicycle length, which characterizes oscillatory solutions that tend to
zero. This follows from both known and new comparison Theorems on the
distance between zeros and local extrema. We obtain an upper bound on the
semicycles guaranteeing that an oscillatory solution is in the stable
manifold, and conversely, a necessary lower bound on the nonoscillation
intervals of unbounded solutions. As we are investigating oscillatory
solutions, the sign of the coefficient may be modified, allowing for
according changes in the delay. Consequently, the case of negative
coefficient is closely related to the case of sign-changing coefficient, and
the bounds on the ascent from zero to maximum give us a bound on the delay
which guarantees that oscillatory solutions tend to zero, under $p(t)\leq
0,t\in 
\mathbb{R}
$. A more detailed discussion of the negative coefficient case is postponed
to Section 5 of the present.

We note that a major difficulty in such estimates are regularity assumptions
on the solutions. The optimal results are expected to correspond to periodic
solutions of \eqref {2.41} (similarly to the first order delay equation).
However, rigorously proving that the bounds for regular solutions apply in
the general case, to irregural solutions as well, is no trivial task. In
addition, investigations without regularity assumptions such as slow
oscillation are rare in the literature. In the present, we reduce the two
dimensional problem to a one dimensional problem by not\ taking into account
the derivative at the zeros, and are thus able to obtain results without any
regularity assumptions.

The main results of this paper are as follows:

\begin{enumerate}
\item Calculation of the threshold semicycle length which distinguishes
between oscillatory solutions that tend to zero, and oscillatory solutions
that may be unbounded.

\item Extension of well-known comparison results and development of new
methods, clarifying the relations between various comparison Theorems.

\item Classification of solutions of \eqref {2.41} with negative
coefficient, including estimates of the critical value of the delay which
guarantees that oscillatory solutions tend to zero.
\end{enumerate}

We first obtain several forward and backward comparison results on %
\eqref{2.41}, as well as estimates on the distances between zeros and
extrema, by comparing with autonomous equations and treating the delay as a
parameter (Sections 2 and 3). We subsequently utilize these growth estimates
to obtain a \textit{continuum} (parametrized by the delay) of lower bounds
on the semicycles of solutions that do not tend to zero (Section 4). Once
the speed of oscillation is higher than this critical value, a solution
necessarily tends to zero. We furthermore obtain new results on the
classification of solutions with negative coefficient (Section 5). Section 6
contains examples illustrating the import and the sharpness of the results,
as well as a discussion of the current and future directions of research.

\section{Descent from a maximum to zero}

Throughout the following, we will assume that 
\begin{equation}
\limfunc{esssup}_{t\in 
\mathbb{R}
}|p(t)|\leq 1.  \label{7.04}
\end{equation}
This assumption cannot harm the generality, in virtue of the following
change of variables, introduced by Myshkis \cite{Myshkis51}.

\begin{lemma}[{\protect\cite[p. 16]{Myshkis51}}]
\label{Lemma3.32}Assume that $x$ solves \eqref {2.41}. For a fixed $k\in
(0,+\infty )$, define the function 
\begin{equation*}
\widetilde{x}(t):=x(kt).
\end{equation*}%
Then $\widetilde{x}$ solves the following equation%
\begin{equation*}
\widetilde{x}^{\prime \prime }(t)+k^{2}p(kt)\widetilde{x}\left( t-\frac{1}{k}%
\tau (kt)\right) =0.
\end{equation*}
\end{lemma}

Let us recall a well-known comparison Theorem. It was proven by Du and Kwong 
\cite{kwong} for continuous parameters, and the same proof applies to the
measurable case, as one retains continuous dependence \cite[Theorem 2.2]%
{Hale}. The following form of the statement is a combination of \cite[%
Theorem 1]{kwong} and \cite[Theorem 2]{kwong}, cf. the discussion in \cite[%
p. 310]{kwong}.

\begin{lemma}[\protect\cite{kwong}]
\label{Lemma7.42}\bigskip Assume that $z$ solves \eqref {2.41} on $%
[0,+\infty )$, and $y$ solves%
\begin{equation*}
y^{\prime \prime }(t)+P(t)y(t-T(t))=0,t\geq 0,
\end{equation*}%
where $P(t)\geq |p(t)|$ and $T(t)\geq \tau (t),t\geq 0$. Further assume 
\begin{equation*}
y(t)>0,t\in \lbrack -T_{m},a),a>0,
\end{equation*}%
where $T_{m}=\limfunc{esssup}\{T(t):t\geq 0\}$, and that 
\begin{equation*}
y^{\prime }(t)\leq 0,t\in \lbrack -T_{m},0].
\end{equation*}%
If $\tau _{m}>0$ assume 
\begin{equation}
|\frac{z(t)}{z(0)}|\leq \frac{y(t)}{y(0)},t\in \lbrack -\tau _{m},0]
\label{4.29}
\end{equation}%
and if $\tau _{m}=0$ assume 
\begin{equation}
\frac{z^{\prime }(0)}{z(0)}\geq \frac{y^{\prime }(0)}{y(0)}.  \label{4.30}
\end{equation}%
Then 
\begin{equation*}
\frac{z(t)}{z(0)}\geq \frac{y(t)}{y(0)}>0,t\in \lbrack 0,a).
\end{equation*}
\end{lemma}

As a consequence of this comparison, descending from a maximum to a zero, a
solution of \eqref {2.41} is bounded from below by an appropriately scaled
multiple of the following autonomous equation {(see \cite{kwong, Myshkis51}).%
}

\begin{lemma}[{\protect\cite[p. 25-28]{Myshkis51}}]
\label{Lemma6.42}\bigskip Consider a fixed $\Delta \in \lbrack 0,+\infty )$
and the solution $r_{\Delta }$ of the equation 
\begin{eqnarray}
r_{\Delta }^{\prime \prime }(t)+r_{\Delta }(t-\Delta ) &=&0,\text{\ }t\geq
0\   \label{4.21} \\
r_{\Delta }(0) &=&1,t\leq 0  \notag \\
r_{\Delta }^{\prime }(0) &=&0  \notag
\end{eqnarray}%
and denote its first zero by $\vartheta _{\Delta }$. Then $r_{\Delta }$ is
strictly decreasing on $[0,\vartheta _{\Delta }]$, and satisfies the
following expression for $\Delta >0$%
\begin{eqnarray*}
r_{\Delta }(t) &=&1-\frac{t^{2}}{2},t\in \lbrack 0,\Delta ] \\
r_{\Delta }(t) &=&1-\frac{t^{2}}{2}+\frac{(t-\Delta )^{4}}{24},t\in \lbrack
\Delta ,2\Delta ] \\
&&...
\end{eqnarray*}%
\begin{equation}
r_{\Delta }(t)=\sum_{k=0}^{n}\frac{(-1)^{k}}{(2k)!}\left[ t-(k-1)\Delta %
\right] ^{2k},t\in \lbrack (n-1)\Delta ,n\Delta ],n=1,2,...  \label{10.39}
\end{equation}%
As $\Delta \rightarrow 0^{+}$, $r_{\Delta }(t)$ tends uniformly to $\cos
(t),t\in \lbrack 0,\frac{\pi }{2}].$ Its first zero $\vartheta _{\Delta
}\geq $ $\sqrt{2}$ is a continuous function of $\Delta \in \lbrack 0,+\infty
)$. \ 
\end{lemma}

\begin{remark}
\label{Remark1}Lemmas {\ref{Lemma7.42}, \ref{Lemma6.42}, imply that under }%
\eqref {7.04}$,$ if $x$ is a solution of \eqref {2.41} on $[0,\vartheta
_{\Delta }]$ such that $|x(t)|\leq x(0)=1,t\in \lbrack -\tau _{m},0],$ and $%
x^{\prime }(0)\geq 0,$ we have 
\begin{equation*}
x(t)\geq r_{\tau _{m}}(t)>0,t\in \lbrack 0,\vartheta _{\tau _{m}}).
\end{equation*}%
Thus{\ the distance in time descending from a maximum to the next zero is at
least} $\vartheta _{\tau _{m}}\geq \sqrt{2}$.
\end{remark}

The following Lemma will help us prove a backward analog of Lemma {\ref%
{Lemma7.42}, utilizing the }Vall\'{e}e Poussin and Sturm Separation Theorems 
\cite{Azb1971}, \cite[Chapter 11]{BDK}, \cite[Lemma 1]{labo}. We remark that
this framework can provide an alternative proof to Lemma {\ref{Lemma7.42}.}

\begin{lemma}[{\protect\cite{Azb1971}, \protect\cite[Lemma 1]{labo}}]
\label{Lemma3.39}Assume there exists a nonnegative solution of 
\begin{eqnarray*}
v^{\prime \prime }(t)+\chi \lbrack a,b](t-\tau (t))p(t)v(t-\tau (t)) &\leq
&0,t\in \lbrack a,b], \\
v(a) &>&0,
\end{eqnarray*}%
where $p(t)\geq 0,t\in \lbrack a,b]$ and $\chi S(\cdot )$ is the
characteristic function of any set $S\subset 
\mathbb{R}
$.\ Then there exists no nontrivial solution of 
\begin{eqnarray*}
u^{\prime \prime }(t)+\chi \lbrack a,b](t-\tau (t))p(t)u(t-\tau (t)) &\geq
&0,t\in \lbrack a,b] \\
u(a) &=&u(b)=0 \\
u^{\prime }(a) &\geq &0.
\end{eqnarray*}
\end{lemma}

\begin{theorem}
\label{Theorem3.37} Assume that $y$ solves%
\begin{equation*}
y^{\prime \prime }(t)+P(t)y(t-T(t))=0,t\geq 0,
\end{equation*}%
where $P(t)\geq |p(t)|$ and $T(t)\geq \tau (t),t\geq 0$. Further assume 
\begin{equation*}
y(t)>0,t\in \lbrack -T_{m},b),b>0,
\end{equation*}%
where $T_{m}=\limfunc{esssup}\{T(t):t\geq 0\}$, and that 
\begin{equation*}
y^{\prime }(t)\leq 0,t\in \lbrack -T_{m},0].
\end{equation*}%
Assume that $z$ solves \eqref {2.41} on $[0,+\infty )$ and let $z,y,$
satisfy 
\begin{eqnarray}
|z(t)| &\leq &y(t),t\in \lbrack -\tau _{m},0],  \label{5.05} \\
|z(b)| &=&y(b).  \label{5.38}
\end{eqnarray}%
\newline
Then%
\begin{equation*}
|z(t)|\leq y(t),t\in \lbrack 0,b].
\end{equation*}
\end{theorem}

\begin{proof}
Assuming the contrary, there exists $\xi \in (0,b)$ such that $y(\xi
)<|z(\xi )|.$ Hence the set 
\begin{equation*}
A:=\{t\in \lbrack -\tau _{m},b):y(t)<|z(t)|\}
\end{equation*}%
is nonempty. Furthermore, by \eqref{5.05}{, we have }$A\subset (0,b)$. We
denote the infinum of this set by $a$, and note that $a\geq 0.$ By
continuity, we have 
\begin{equation*}
y(a)\leq |z(a)|.
\end{equation*}%
If $a>0$ we must have 
\begin{equation}
y(a)=|z(a)|,  \label{5.06}
\end{equation}%
by continuity and the definition of $a$, and if $a=0,$ we again have %
\eqref{5.06}, because of \eqref{5.05}$.$

As \eqref{5.06} implies $a\notin A$, there exists a sequence of points $%
t_{n}\in A$, such that $t_{n}\overset{n\rightarrow \infty }{\rightarrow }%
a^{+}.$ Hence, 
\begin{equation}
\frac{y(t_{n})}{y(a)}<|\frac{z(t_{n})}{z(a)}|\text{.}  \label{5.07}
\end{equation}%
Relations \eqref{5.06}, \eqref{5.07}, imply 
\begin{equation}
\frac{y^{\prime }(a)}{y(a)}\leq \frac{z^{\prime }(a)}{z(a)}\text{.}
\label{5.37}
\end{equation}

Applying Lemma {\ref{Lemma7.42}, we have that }%
\begin{equation*}
\frac{z(t)}{z(a)}\geq \frac{y(t)}{y(a)}>0,t\in \lbrack a,b).
\end{equation*}%
We may assume $z(t)>0,t\in \lbrack a,b)$, without loss of generality
(otherwise, we consider $-z$). For $t\in \lbrack a,b]$, we have by
monotonicity, positivity, of $y$,

\begin{equation}
y^{\prime \prime }(t)+|p(t)|y(t-\tau (t))\leq y^{\prime \prime
}(t)+P(t)y(t-\tau (t))\leq y^{\prime \prime }(t)+P(t)y(t-T(t))=0,
\label{8.24}
\end{equation}%
and by the nonnegativity of $z,$%
\begin{equation}
z^{\prime \prime }(t)+\chi \lbrack a,b](t-\tau (t))|p(t)|z(t-\tau (t))\geq
-\chi \lbrack a-\tau _{m},a](t-\tau (t))|p(t)||z(t-\tau (t))|.  \label{8.25}
\end{equation}%
Setting $u(t):=z(t)-y(t),t\in \lbrack a,b]$ and subtracting \eqref{8.24}
from \eqref {8.25}%
\begin{eqnarray*}
u^{\prime \prime }(t)+\chi \lbrack a,b](t-\tau (t))|p(t)|u(t-\tau (t)) &\geq
&\chi \lbrack a-\tau _{m},a](t-\tau (t))|p(t)|\left( y(t-\tau (t))-|z(t-\tau
(t))|\right) \\
&\geq &0,t\in \lbrack a,b]
\end{eqnarray*}%
where $\chi A(\cdot )$ is the characteristic function of any set $A\subset 
\mathbb{R}
$. Furthermore, in virtue of \eqref {5.38} and \eqref {5.37}, 
\begin{eqnarray*}
u(a) &=&u(b)=0 \\
u^{\prime }(a) &\geq &0.
\end{eqnarray*}%
It suffices to show that the assumptions of Lemma {\ref{Lemma3.39} are
satisfied. By \eqref{8.24}, we have that }$y=v$ is the required nonnegative
solution of Lemma {\ref{Lemma3.39}.}
\end{proof}

The following Corollary immediately follows, and a direct proof solely using
Lemma {\ref{Lemma7.42} is also possible.}

\begin{corollary}
\label{Corollary5.40}For a fixed $\Delta \in \lbrack 0,+\infty ),$ consider
any solution $x$ of \eqref {2.41} on $[0,\vartheta _{\Delta }]$ with %
\eqref{7.04}$,$ $\tau _{m}\leq \Delta $ and $x(\vartheta _{\Delta })=0.$
Then 
\begin{equation*}
\left( \max_{t\in \lbrack -\tau _{m},0]}|x(t)|\right) r_{\Delta }(t)\geq
|x(t)|,t\in \lbrack 0,\vartheta _{\Delta }].
\end{equation*}%
\qquad \qquad
\end{corollary}

\begin{corollary}
$\vartheta _{\Delta }$ is a nonincreasing function of $\Delta $. More
precisely, $\vartheta _{\Delta }=\sqrt{2},\forall \Delta \geq \sqrt{2}$, and 
$\vartheta _{\Delta }$ is strictly decreasing from $\frac{\pi }{2}$ to $%
\sqrt{2}$ for $\Delta \in \lbrack 0,\sqrt{2}]$.
\end{corollary}

\begin{proof}
The nonincreasing nature of $\vartheta _{\Delta }$ follows from Lemma {\ref%
{Lemma7.42}, setting }$y=r_{\Delta },z=r_{\widetilde{\Delta }}$ with $\Delta
\geq \widetilde{\Delta }$. That $\vartheta _{\Delta }=\sqrt{2},\forall
\Delta \geq \sqrt{2}$ is an immediate consequence of \eqref {10.39}. Now
assume that $\vartheta _{\widetilde{\Delta }}=\vartheta _{\Delta }$ where $0<%
\widetilde{\Delta }<\Delta \leq \sqrt{2}$. By Corollary {\ref{Corollary5.40}%
, }Lemma {\ref{Lemma7.42} (setting }$y=r_{\Delta },z=r_{\widetilde{\Delta }%
}),$ we have $r_{\Delta }(t)=r_{\widetilde{\Delta }}(t),t\in \lbrack
0,\vartheta _{\Delta }]$. Expression \eqref {10.39} and $0<\widetilde{\Delta 
}<\Delta \leq \sqrt{2}$ give a contradiction.
\end{proof}

\section{\protect\bigskip Ascent from a zero to a maximum}

We now turn our attention to the ascent from a zero to the next extremum. We
reduce the two dimensional problem to a one dimensional problem, without
taking into account the first derivative at the zero. This reduction allows
us to obtain results that apply with arbitrary scaling and semicycle length.
It furthermore allows us to directly combine the estimates on the ascent and
on the descent, without imposing any regularity restrictions -such as slow
oscillation- on the oscillatory solutions. For this reason, it is simplest
to reverse time and consider the independent variable $w:=-t$, with respect
to which our reasoning is similar to the previous section. In order to
estimate the least distance in time necessary to ascend from a zero to a
maximum, we need to define a sequence of upper bounds, constructed so that
it converges to a solution which ascends to a maximum in the least time.
Throughout the following a fixed $\Delta \in [ 0,+\infty )$ will denote a
bound on the maximum delay, and $\rho \in (0,+\infty )$ will be a scaling
constant.

\begin{lemma}
\label{Lemma5.03}For fixed $\Delta \in \lbrack 0,+\infty ),\rho \in
(0,+\infty )$ consider the following sequence of functions. Setting $\beta
_{0}(t)\equiv 1$, we define for $n=0,1,2,...$%
\begin{equation}
\beta _{n+1}(t)=\left\{ 
\begin{array}{cc}
1, & t\leq -\varpi _{n} \\ 
\displaystyle1-\int_{-\varpi _{n}}^{t}\left[ \int_{-\varpi _{n}}^{s}\max
\{\beta _{n}(v),\rho r_{\Delta }(\vartheta _{\Delta }-v-\Delta )\chi \lbrack
-\Delta ,0](v)\}dv\right] ds, & t\in \lbrack -\varpi _{n},0]%
\end{array}%
\right.  \label{7.01}
\end{equation}%
where $\varpi _{n}\in (0,+\infty )$, $n=0,1,2,...$is the unique solution of 
\begin{equation*}
h_{n}(-\varpi _{n})=\int_{-\varpi _{n}}^{0}\left[ \int_{-\varpi
_{n}}^{s}\max \{\beta _{n}(w),\rho r_{\Delta }(\vartheta _{\Delta }-w-\Delta
)\chi \lbrack -\Delta ,0](w)\}dw\right] ds=1,
\end{equation*}%
where $h_{n}:(-\infty ,0]\rightarrow \lbrack 0,+\infty )$ is the function 
\begin{equation}
h_{n}(v):=\int_{v}^{0}\left[ \int_{v}^{s}\max \{\beta _{n}(w),\rho r_{\Delta
}(\vartheta _{\Delta }-w-\Delta )\chi \lbrack -\Delta ,0](w)\}dw\right] ds
\label{4.53}
\end{equation}%
and $\chi A(\cdot )$ is the characteristic function of any set $A\subset 
\mathbb{R}
$. The sequence $\beta _{n}$ is well-defined and pointwise nonincreasing,
and the sequence $\varpi _{n}$ is nondecreasing.
\end{lemma}

\begin{proof}
Assuming that $\beta _{n}$ is well-defined for $n\leq k-1$, where $k$ is a
positive integer, we shall show that $\beta _{k}$ is well-defined. By 
\eqref
{7.01}, and the definition of $\beta _{0}(t)\equiv 1$, we have 
\begin{equation}
\beta _{k-1}(t)>0,t<0,  \label{4.49}
\end{equation}%
and for a sufficiently large $K>0,$%
\begin{equation}
\beta _{k-1}(t)=1,t<-K.  \label{4.50}
\end{equation}%
By \eqref {4.49}, the function $h_{k-1}$ is continuous strictly decreasing.
By \eqref {4.50}, we furthermore have 
\begin{equation*}
h_{k-1}(-K-\sqrt{2})>1.
\end{equation*}%
Hence, there must exist a unique $\varpi _{k-1}\in (0,+\infty )$ such that $%
h_{k-1}(-\varpi _{k-1})=1$. This shows that $b_{k\text{ }}$ is well-defined,
as in \eqref {7.01}. We also note that by the definition of $\varpi _{n},$
each $\beta _{n}$ is strictly decreasing on $[-\varpi _{n-1},0]$, $n=1,2,...$%
.

We now show that 
\begin{equation}
\beta _{n+1}(t)\leq \beta _{n}(t),t\leq 0.  \label{4.57}
\end{equation}

Notice that, by definition \eqref {7.01}, inequality \eqref {4.57} implies $%
\varpi _{n}\leq \varpi _{n+1}$. For $n=0$, the inequality \eqref {4.57} is
obvious, considering \eqref {7.01}. Assuming \eqref {4.57} for $n\leq k-1,$
where $k$ is a positive integer$,$ we will prove \eqref {4.57} for $n=k$.
Notice that by \eqref{4.57} for $n=k-1,$ and definition \eqref {7.01},%
\begin{equation}
\begin{array}{ll}
\beta _{k+1}^{\prime \prime }(t) & \geq \displaystyle-\max \{\beta
_{k}(t),\rho r_{\Delta }(\vartheta _{\Delta }-w-\Delta )\chi \lbrack -\Delta
,0](t)\}\vspace{2mm} \\ 
& \geq \displaystyle-\max \{\beta _{k-1}(t),\rho r_{\Delta }(\vartheta
_{\Delta }-w-\Delta )\chi \lbrack -\Delta ,0](t)\}=\beta _{k}^{\prime \prime
}(t),t\in \lbrack -\varpi _{k-1},0].%
\end{array}
\label{4.58}
\end{equation}%
Inequality \eqref {4.58} together with 
\begin{eqnarray*}
\beta _{k+1}(0) &=&\beta _{k}(0)=0 \\
1 &=&\beta _{k}(-\varpi _{k-1})\geq \beta _{k+1}(-\varpi _{k-1})
\end{eqnarray*}%
give \eqref {4.57} for $n=k$. In fact, $\ f(t):=\beta _{k}(t)-\beta
_{k+1}(t),t\in \lbrack -\varpi _{k-1},0]$ is concave ($f^{\prime \prime
}(t)\leq 0$) and satisfies $f(-\varpi _{k-1})\geq 0=f(0)$. Therefore the
graph of $f$ lies above its nonnegative chord, and $\beta _{k}(t)\geq \beta
_{k+1}(t),t\in \lbrack -\varpi _{k-1},0]$.
\end{proof}

For fixed $\Delta \in [ 0,+\infty ),\rho \in (0,+\infty )$, we have shown
the sequence $\varpi _{n}$ of Lemma {\ref{Lemma5.03} }is nondecreasing. We
shall denote the limit of $\varpi _{n}$ as $n\rightarrow \infty $ by 
\begin{equation}
\Psi (\rho ,\Delta ):=\varpi _{\infty }=\lim_{n\rightarrow \infty }\varpi
_{n}.  \label{7.11}
\end{equation}%
The following Theorem shows that this limit is bounded by $\frac{\pi }{2}$.

\begin{theorem}
\label{Theorem8.24} For fixed $\Delta \in \lbrack 0,+\infty ),\rho \in
(0,+\infty )$, the limit $\Psi (\rho ,\Delta )$, as in \eqref {7.11},
satisfies 
\begin{equation*}
\Psi (\rho ,\Delta )\leq \frac{\pi }{2}.
\end{equation*}%
The sequence of functions $\beta _{n}(t),$ as in Lemma {\ref{Lemma5.03}, }
converges to a solution $y_{\rho ,\Delta }$ of the following boundary value
problem%
\begin{eqnarray}
y_{\rho ,\Delta }^{\prime \prime }(w)+\max \{y_{\rho ,\Delta }(w),\rho
r_{\Delta }(\vartheta _{\Delta }-w-\Delta )\chi \lbrack -\Delta ,0](w)\}
&=&0,w\in \lbrack -\Psi (\rho ,\Delta ),0]  \label{6.35} \\
y_{\rho ,\Delta }(0) &=&0<y_{\rho ,\Delta }(w),w\in \lbrack -\Psi (\rho
,\Delta ),0)  \notag \\
y_{\rho ,\Delta }(w) &=&1,w\leq -\Psi (\rho ,\Delta )  \notag \\
y_{\rho ,\Delta }^{\prime }(-\Psi (\rho ,\Delta )) &=&0  \notag
\end{eqnarray}%
For any nonnegative function $x$ which satisfies 
\begin{eqnarray}
x(0) &=&0  \notag \\
\max_{t\in \lbrack -\Psi (\rho ,\Delta ),0]}x(t) &\leq &1  \notag \\
x^{\prime \prime }(w)+\max \{\max_{t\in \lbrack w,0]}x(t),\rho r_{\Delta
}(\vartheta _{\Delta }-w-\Delta )\chi \lbrack -\Delta ,0](w)\} &\geq &0,w\in
\lbrack -\Psi (\rho ,\Delta ),0]  \label{8.46}
\end{eqnarray}%
we have the following bound%
\begin{equation}
y_{\rho ,\Delta }(w)\geq x(w),w\in \lbrack -\Psi (\rho ,\Delta ),0].
\label{4.22}
\end{equation}
\end{theorem}

\begin{proof}
\bigskip Let us show that 
\begin{equation}
z(t)\leq \beta _{n}(t),t\leq 0  \label{4.14}
\end{equation}%
where $\displaystyle z(t):=\left\{ 
\begin{array}{ll}
1 & ,t<-\frac{\pi }{2}, \\ 
\cos (t+\frac{\pi }{2}), & t\in \lbrack -\frac{\pi }{2},0].%
\end{array}%
\right. $ For $n=0$, the inequality \eqref {4.14} is obvious. Assuming \eqref%
{4.14} for $n\leq k-1,$ where $k$ is a positive integer$,$ we will prove %
\eqref {4.14} for $n=k$. Notice that by \eqref{4.14} for $n=k-1,$ and
definition \eqref {7.01}, 
\begin{equation}
\begin{array}{c}
\beta _{k}^{\prime \prime }(t)=-\max \{\beta _{k-1}(t),\rho r_{\Delta
}(\vartheta _{\Delta }-w-\Delta )\chi \lbrack -\Delta ,0](t)\} \\ 
\leq -\max \{z(t),\rho r_{\Delta }(\vartheta _{\Delta }-w-\Delta )\chi
\lbrack -\Delta ,0](t)\}\leq z^{\prime \prime }(t),t\in \lbrack -\varpi
_{k-1},0].%
\end{array}
\label{9.01}
\end{equation}%
Inequality \eqref {9.01} together with 
\begin{eqnarray*}
z(0) &=&\beta _{k}(0)=0 \\
1 &=&\beta _{k}(-\varpi _{k-1})\geq z(-\varpi _{k-1}),
\end{eqnarray*}%
give \eqref {4.14} for $n=k$. In fact, $\ f(t):=\beta _{k}(t)-z(t),t\in
\lbrack -\varpi _{k-1},0]$ is concave ($f^{\prime \prime }(t)\leq 0$) and
satisfies $f(-\varpi _{k-1})\geq 0=f(0)$. Therefore the graph of $f$ lies
above its nonnegative chord, and $\beta _{k}(t)\geq z(t),t\in \lbrack
-\varpi _{k-1},0]$.

Integrating \eqref {4.14}, we obtain%
\begin{equation*}
\int_{-\frac{\pi }{2}}^{0}\left[ \int_{-\frac{\pi }{2}}^{s}\max \{\beta
_{n}(w),\rho r_{\Delta }(\vartheta _{\Delta }-w-\Delta )\chi \lbrack -\Delta
,0](w)\}dw\right] ds\geq \int_{-\frac{\pi }{2}}^{0}\left[ \int_{-\frac{\pi }{%
2}}^{s}\cos (w+\frac{\pi }{2})dw\right] ds=1,
\end{equation*}%
hence \eqref {4.14} implies $\frac{\pi }{2}\geq \varpi _{n}$.

By \eqref {4.14}, Lemma {\ref{Lemma5.03}}, we have $\varpi _{n}\leq \varpi
_{n+1}\leq ...\leq \varpi _{\infty }\leq $ $\frac{\pi }{2}$. Finally, notice
that the sequence $\beta _{n}(t)$ is a nonincreasing sequence of functions
that, by dominated convergence, converges to $y_{\rho ,\Delta }$ pointwise.

The proof of \eqref {4.22} is similar to the proof of \eqref{4.14}, for we
need only show 
\begin{equation}
x(t)\leq \beta _{n}(t),t\in \lbrack -\Psi (\rho ,\Delta ),0].  \label{3.58}
\end{equation}%
For $n=0$, the inequality \eqref {3.58} is obvious. Assuming \eqref {3.58}
for $n\leq k-1,$ where $k$ is a positive integer$,$ we will prove %
\eqref{3.58} for $n=k$. Notice that by \eqref {3.58} for $n=k-1,$ definition 
\eqref
{7.01}, and the monotonicity of $\beta _{k-1}$, 
\begin{equation}
\begin{array}{c}
\beta _{k}^{\prime \prime }(t)=-\max \{\beta _{k-1}(t),\rho r_{\Delta
}(\vartheta _{\Delta }-w-\Delta )\chi \lbrack -\Delta ,0](t)\} \\ 
\leq -\max \{\max_{a\in \lbrack t,0]}x(a),\rho r_{\Delta }(\vartheta
_{\Delta }-w-\Delta )\chi \lbrack -\Delta ,0](t)\}\leq x^{\prime \prime
}(t),t\in \lbrack -\varpi _{k-1},0].%
\end{array}
\label{4.00}
\end{equation}%
Inequality \eqref {4.00} together with 
\begin{eqnarray*}
x(0) &=&\beta _{k}(0)=0 \\
1 &=&\beta _{k}(-\varpi _{k-1})\geq x(-\varpi _{k-1}),
\end{eqnarray*}%
give \eqref {3.58} for $n=k$. In fact, $\ f(t):=\beta _{k}(t)-x(t),t\in
\lbrack -\varpi _{k-1},0]$ is concave ($f^{\prime \prime }(t)\leq 0$) and
satisfies $f(-\varpi _{k-1})\geq 0=f(0)$. Therefore the graph of $f$ lies
above its nonnegative chord, and $\beta _{k}(t)\geq x(t),t\in \lbrack
-\varpi _{k-1},0]$.
\end{proof}

Inequality \eqref {4.22} shows that $\Psi (\rho ,\Delta )$ is the least
distance in time necessary for a solution of \eqref {8.46} to descend to
zero. By reversing time, we obtain an estimate of the least distance
necessary for a solution of \eqref {2.41} to ascend from a zero to a maximum.

\begin{theorem}
\bigskip \label{Theorem10.28}Consider any solution $x$ of \eqref {2.41} on $%
[-\vartheta _{\Delta },+\infty )$ with \eqref {7.04}$,$ $\tau _{m}\leq
\Delta ,$ such that 
\begin{eqnarray*}
\max_{t\in [ -\Delta -\vartheta _{\Delta },0]}|x(t)| &\leq &\rho \in
(0,+\infty ) \\
x(0) &=&0 \\
1 &\geq &x(t)\geq 0,t\in [ 0,A] \\
x(A) &=&1,x^{\prime }(A)=0
\end{eqnarray*}%
Then, with $\Psi (\rho ,\Delta )$ as in \eqref {7.11}, we have $\Psi (\rho
,\Delta )\leq A$ and 
\begin{equation}
x(t)\leq y_{\rho ,\Delta }(-t),t\in [ 0,\Psi (\rho ,\Delta )].  \label{2.30}
\end{equation}
\end{theorem}

\begin{proof}
Applying Corollary {\ref{Corollary5.40}, we have }$|x(t)|\leq \rho r_{\Delta
}(\vartheta _{\Delta }+t),t\in \lbrack -\Delta -\vartheta _{\Delta },0].$
The rest follows from Theorem {\ref{Theorem8.24}, condition }\eqref {7.04},{%
\ by reversing time. In fact, setting }$w:=-t$, the function $\displaystyle %
f(w):=\left\{ 
\begin{array}{ll}
1, & w<-A, \\ 
x(-w), & w\in \lbrack -A,0]%
\end{array}%
\right. $ satisfies for $w\in \lbrack -\Psi (\rho ,\Delta ),0],$ where $\Psi
(\rho ,\Delta )$ is defined in\eqref {7.11},%
\begin{equation*}
\begin{array}{ll}
f^{\prime \prime }(w) & \displaystyle=\left\{ 
\begin{array}{ll}
x^{\prime \prime }(-w), & w\in \lbrack -A,0]\cap \lbrack -\Psi (\rho ,\Delta
),0] \\ 
0, & w\in (-\infty ,-A)\cap \lbrack -\Psi (\rho ,\Delta ),0]%
\end{array}%
\right. \\ 
& \displaystyle\geq -\max \{\max_{u\in \lbrack w,0]}f(u),\rho r_{\Delta
}(\vartheta _{\Delta }-w-\Delta )\chi \lbrack -\Delta ,0](w)\}.%
\end{array}%
\end{equation*}%
Assuming $\Psi (\rho ,\Delta )>A$ we obtain a contradiction by Theorem {\ref%
{Theorem8.24} and }the strict monotonicity of $y_{\rho ,\Delta }$. We may
conclude that \eqref {2.30} follows from Theorem {\ref{Theorem8.24}.}
\end{proof}

\begin{corollary}
\label{Corollary5.41}Consider any solution $x$ of \eqref {2.41} with %
\eqref{7.04}$,$ $\tau _{m}\leq \Delta ,$ that has a semicycle $(a,b)$. Fix
any $\rho \geq \max_{t\in \lbrack -\Delta -\vartheta _{\Delta }+a,a]}|x(t)|$%
. Theorem {\ref{Theorem10.28} implies that the extremum of this semicycle is
attained at points }$w$ such that 
\begin{equation*}
w-a\geq \Psi \left( \frac{\rho }{\max_{t\in \lbrack a,b]}|x(t)|},\Delta
\right) ,
\end{equation*}%
where $\Psi (\rho ,\Delta )$ is defined in \eqref {7.11}.
\end{corollary}

\begin{lemma}
\label{Lemma2.47}$\Psi (\rho ,\Delta )$, defined in \eqref {7.11}{,} is a
nonincreasing continuous function of $\Delta \in \lbrack 0,+\infty )$ for
fixed $\rho \in (0,+\infty )$, strictly decreasing and continuous with
respect to $\rho $ for fixed $\Delta >0.$ Furthermore, $\Psi (\rho ,0)=\frac{%
\pi }{2}$, and $\Psi (1,\Delta )=\sqrt{2}$ for $\Delta \geq 2\sqrt{2}$.
\end{lemma}

\begin{proof}
By application of Theorem {\ref{Theorem8.24}, Corollary \ref{Corollary5.40}, 
}$\Psi (\rho ,\Delta )$ is nonincreasing with respect to $\Delta $,
nonincreasing with respect to $\rho $.

Assuming that {\ }$\Psi (\rho ,\Delta )=${\ }$\Psi (\widetilde{\rho },\Delta
)$ with $\widetilde{\rho }>\rho >0$, by Theorem {\ref{Theorem8.24}}, $y_{%
\widetilde{\rho },\Delta }(w)\geq y_{\rho ,\Delta }(w),w\in \lbrack -\Psi
(\rho ,\Delta ),0].$Hence by \eqref{6.35}, $y_{\widetilde{\rho },\Delta
}^{\prime \prime }(w)\leq y_{\rho ,\Delta }^{\prime \prime }(w),w\in \lbrack
-\Psi (\rho ,\Delta ),0]$ and we conclude%
\begin{equation}
y_{\widetilde{\rho },\Delta }(w)=y_{\rho ,\Delta }(w),w\in \lbrack -\Psi
(\rho ,\Delta ),0].  \label{2.53}
\end{equation}
If $\Delta >0$, for $w$ approaching $0$ from the left, 
\begin{equation*}
y_{\widetilde{\rho },\Delta }^{\prime \prime }(w)=-\widetilde{\rho }%
r_{\Delta }(\vartheta _{\Delta }-w-\Delta )\chi \lbrack -\Delta ,0](w)<-\rho
r_{\Delta }(\vartheta _{\Delta }-w-\Delta )\chi \lbrack -\Delta
,0](w)=y_{\rho ,\Delta }^{\prime \prime }(w),
\end{equation*}%
contradicting \eqref{2.53}. Therefore $\Psi (\rho ,\Delta )$ is strictly
increasing with respect to $\rho $ for fixed $\Delta >0$.

Let us prove continuity. We recall that we have uniform continuous
dependence \cite[Theorem 2.2]{Hale}. By applying Theorem {\ref{Theorem8.24}
with }$x(w)=\cos (\frac{\pi }{2}+w)$,$w\in [ -\frac{\pi }{2},0]$ we also
have $y_{\rho ,\Delta }^{\prime }(0)\leq \cos ^{\prime }(\frac{\pi }{2})=-1$.

Assume there exists a left discontinuity with respect to $\rho $. In other
words, assume there exists an increasing sequence $\rho _{n}<\rho
_{n+1}<...<\rho _{\infty }$ such that the limit function%
\begin{equation*}
\varphi (w):=\lim_{n\rightarrow \infty }y_{\rho _{n},\Delta }(w),w\in
\lbrack -\frac{\pi }{2},0]
\end{equation*}%
does not equal $y_{\rho _{\infty },\Delta }$. Notice that this function is
well defined, and $y_{\rho _{\infty },\Delta }(w)\geq \varphi (w),w\in
\lbrack -\frac{\pi }{2},0]$, in virtue of the nonincreasing nature of $\Psi $%
, and Theorem {\ref{Theorem8.24}.} We have that $z_{n}:=\frac{\rho _{n}}{%
\rho _{\infty }}y_{\rho _{\infty },\Delta }$ satisfies 
\begin{eqnarray*}
z_{n}^{\prime \prime }(w)+\max \{z_{n}(w),\rho _{n}r_{\Delta }(\vartheta
_{\Delta }-w-\Delta )\chi \lbrack -\Delta ,0](w)\} &\geq &0,w\in \lbrack
-\Psi (\rho _{n},\Delta ),0] \\
\frac{\rho _{n}}{\rho _{\infty }} &\geq &z_{n}(w)\geq 0,w\in \lbrack -\Psi
(\rho _{n},\Delta ),0]
\end{eqnarray*}%
By Theorem {\ref{Theorem8.24}, }$\varphi (w)\geq z_{n}(w),w\in \lbrack -\Psi
(\rho _{n},\Delta ),0]$. Now notice that $y_{\rho _{\infty },\Delta }$ is
the limit of $z_{n},$ as $n\rightarrow \infty $.

Now assume a right discontinuity with respect to $\rho $. In other words,
assume there exists a decreasing sequence $\rho _{n}>\rho _{n+1}>...>\rho
_{\infty }$ such that the limit function 
\begin{equation*}
\varphi (w):=\lim_{n\rightarrow \infty }y_{\rho _{n},\Delta }(w),w\in
\lbrack -\frac{\pi }{2},0]
\end{equation*}%
does not equal $y_{\rho _{\infty },\Delta }$. Notice that this function is
well defined, and $y_{\rho _{\infty },\Delta }(w)\leq \varphi (w),w\in
\lbrack -\frac{\pi }{2},0]$, in virtue of the nonincreasing nature of $\Psi $%
, and Theorem {\ref{Theorem8.24}. }We have that $z_{n}:=\frac{\rho _{\infty }%
}{\rho _{n}}y_{\rho _{n},\Delta }$ satisfies 
\begin{eqnarray*}
z_{n}^{\prime \prime }(w)+\max \{z_{n}(w),\rho _{\infty }r_{\Delta
}(\vartheta _{\Delta }-w-\Delta )\chi \lbrack -\Delta ,0](w)\} &\geq &0,w\in
\lbrack -\Psi (\rho _{_{\infty }},\Delta ),0] \\
\frac{\rho _{\infty }}{\rho _{n}} &\geq &z_{n}(w)\geq 0,w\in \lbrack -\Psi
(\rho _{\infty },\Delta ),0]
\end{eqnarray*}%
By Theorem {\ref{Theorem8.24}, }$y_{\rho _{\infty },\Delta }(w)\geq
z_{n}(w),w\in \lbrack -\Psi (\rho _{\infty },\Delta ),0]$. Now notice that $%
\varphi $ is the limit of $z_{n},$ as $n\rightarrow \infty $.

Assume there exists a left discontinuity with respect to $\Delta $. In other
words, assume there exists an increasing sequence $\Delta _{n}<\Delta
_{n+1}<...<\Delta _{\infty }$ such that the limit function 
\begin{equation*}
\varphi (w):=\lim_{n\rightarrow \infty }y_{\rho ,\Delta _{n}}(w),w\in
\lbrack -\frac{\pi }{2},0]
\end{equation*}%
does not equal $y_{\rho ,\Delta _{\infty }}$. Notice that this function is
well defined, and $y_{\rho ,\Delta _{\infty }}(w)\geq \varphi (w),w\in
\lbrack -\frac{\pi }{2},0]$, in virtue of the nonincreasing nature of $\Psi $%
, and Theorem {\ref{Theorem8.24}.} For any fixed $\varepsilon >0$ there
exists a $K(\varepsilon )\in 
\mathbb{N}
$ such that $n>K(\varepsilon )$ implies that the solution $z_{n}$ of 
\begin{eqnarray*}
z_{n}^{\prime \prime }(w)+\max \{z_{n}(w),\rho r_{\Delta _{n}}(\vartheta
_{\Delta _{n}}-w-\Delta _{n})\chi \lbrack -\Delta _{n},0](w)\} &=&0,w\in
\lbrack -\Psi (\rho ,\Delta _{\infty }),0] \\
z_{n}^{\prime \prime }(w) &=&0,w\leq -\Psi (\rho ,\Delta _{\infty }) \\
z_{n}(0) &=&0,z_{n}^{\prime }(0)=y_{\rho ,\Delta _{\infty }}^{\prime }(0)
\end{eqnarray*}%
satisfies $|z_{n}(w)-y_{\rho ,\Delta _{\infty }}(w)|+|z_{n}^{\prime
}(w)-y_{\rho ,\Delta _{\infty }}^{\prime }(w)|<\varepsilon ,w\in \lbrack -%
\frac{\pi }{2},0]$. Because $y_{\rho ,\Delta _{\infty }}^{\prime }(0)\leq -1$%
, for sufficiently small $\varepsilon $, we also have $z_{n}(w)\geq 0,w\in
\lbrack -\frac{\pi }{2},0]$. Applying Theorem {\ref{Theorem8.24}, }$\varphi
(w)\geq \frac{1}{1+\varepsilon }z_{n}(w),w\in \lbrack -\Psi (\rho
_{n},\Delta ),0]$. Now notice that $y_{\rho ,\Delta _{\infty }}$ is the
limit of $z_{n},$ as $\varepsilon \rightarrow 0$ and $K(\varepsilon
)<n\rightarrow \infty $.

Assume there exists a right discontinuity with respect to $\Delta $. In
other words, assume there exists an increasing sequence $\Delta _{n}>\Delta
_{n+1}>...>\Delta _{\infty }$ such that the limit function 
\begin{equation*}
\varphi (w):=\lim_{n\rightarrow \infty }y_{\rho ,\Delta _{n}}(w),w\in
\lbrack -\frac{\pi }{2},0]
\end{equation*}%
does not equal $y_{\rho ,\Delta _{\infty }}$. Notice that this function is
well defined, and $y_{\rho ,\Delta _{\infty }}(w)\leq \varphi (w),w\in
\lbrack -\frac{\pi }{2},0]$, in virtue of the nonincreasing nature of $\Psi $%
, and Theorem {\ref{Theorem8.24}.} For any fixed $\varepsilon >0$ there
exists a $K(\varepsilon )\in 
\mathbb{N}
$ such that $n>K(\varepsilon )$ implies that the solution $z_{n}$ of 
\begin{eqnarray*}
z_{n}^{\prime \prime }(w)+\max \{z_{n}(w),\rho r_{\Delta _{\infty
}}(\vartheta _{\Delta _{\infty }}-w-\Delta _{\infty })\chi \lbrack -\Delta
_{\infty },0](w)\} &=&0,w\in \lbrack -\Psi (\rho ,\Delta _{n}),0] \\
z_{n}^{\prime \prime }(w) &=&0,w\leq -\Psi (\rho ,\Delta _{n}) \\
z_{n}(0) &=&0,z_{n}^{\prime }(0)=y_{\rho ,\Delta _{n}}^{\prime }(0)
\end{eqnarray*}%
satisfies $|z_{n}(w)-y_{\rho ,\Delta _{n}}(w)|+|z_{n}^{\prime }(w)-y_{\rho
,\Delta _{n}}^{\prime }(w)|<\varepsilon ,w\in \lbrack -\frac{\pi }{2},0]$.
Because $y_{\rho ,\Delta _{n}}^{\prime }(0)\leq -1$, for sufficiently small $%
\varepsilon $, we also have $z_{n}(w)\geq 0,w\in \lbrack -\frac{\pi }{2},0]$%
. Applying Theorem {\ref{Theorem8.24}, }$y_{\rho ,\Delta _{\infty }}(w)\geq 
\frac{1}{1+\varepsilon }z_{n}(w),w\in \lbrack -\Psi (\rho ,\Delta _{\infty
}),0]$. Now notice that $\varphi $ is the limit of $z_{n},$ as $\varepsilon
\rightarrow 0$ and $K(\varepsilon )<n\rightarrow \infty $.

Finally, when $\Delta =0,$ we have $y_{\rho ,0}(w)=\cos (\frac{\pi }{2}%
+w),w\in \lbrack -\frac{\pi }{2},0]$ and when $\rho =1,\Delta \geq 2\sqrt{2}$%
, using \eqref {10.39} and $\vartheta _{\Delta }=\sqrt{2}$, we have by
direct calculation $b_{n}(w)=1-\frac{1}{2}(w+\sqrt{2})^{2},w\in \lbrack -%
\sqrt{2},0]$ for $n=1,2,...$
\end{proof}

\section{Bounds on semicycle length}

Combining the estimates of the previous sections, we can now prove the main
results, calculating lower bounds on the semicycle length of oscillatory
solutions which do not tend to zero.

\begin{theorem}
\label{Theorem10.41}Consider an oscillatory solution $x$\ of \eqref {2.41},
with \eqref {7.04}$.$ When $\tau _{m}>0,$ assume that its semicycles satisfy%
\begin{equation*}
x(t)\neq 0,t\in (a,b)\Longrightarrow b-a\leq \Psi (1,\tau _{m})+\vartheta
_{\tau _{m}}
\end{equation*}%
where $\Psi (\rho ,\Delta )$ is as in \eqref {7.11}, $\vartheta _{(\cdot )%
\text{ }}$as in Lemma {\ref{Lemma6.42},} and when $\tau _{m}=0$, assume that 
\begin{equation*}
x(t)\neq 0,t\in (a,b)\Longrightarrow b-a<\pi .
\end{equation*}%
Then, $x$ is bounded. If, further, for a real constant $c\in (0,\Psi (1,\tau
_{m})+\vartheta _{\tau _{m}}),$ 
\begin{equation}
x(t)\neq 0,t\in (a,b)\Longrightarrow \Psi (1,\tau _{m})+\vartheta _{\tau
_{m}}>c>b-a  \label{10.21}
\end{equation}%
then $x$ tends to zero at infinity.
\end{theorem}

\begin{proof}
Let us consider a point $\xi $ such that $x(\xi )=0$. We shall show that 
\begin{equation}
|x(u)|\leq \max_{t\in \lbrack \xi -\tau _{m}-\vartheta _{\tau _{m}},\xi
]}|x(t)|,u\geq \xi ,  \label{9.46}
\end{equation}%
where $\vartheta _{(\cdot )\text{ }}$ is as in Lemma {\ref{Lemma6.42}. }We
may assume $\max_{t\in \lbrack \xi -\tau _{m}-\vartheta _{\tau _{m}},\xi
]}|x(t)|>0$ without loss of generality. Assuming the contrary, that %
\eqref{9.46} does not hold, by considering the point 
\begin{equation*}
\psi :=\inf \{u>\xi :|x(u)|>\max_{t\in \lbrack \xi -\tau _{m}-\vartheta
_{\tau _{m}},\xi ]}|x(t)|\},
\end{equation*}%
it is easy to see that this point is in a semicycle $(a,b)$ such that 
\begin{equation}
\max_{t\in (a,b)}|x(t)|>\max_{t\in \lbrack \xi -\tau _{m}-\vartheta _{\tau
_{m}},a]}|x(t)|=\max_{t\in \lbrack \xi -\tau _{m}-\vartheta _{\tau _{m}},\xi
]}|x(t)|,  \label{5.44}
\end{equation}%
where $a\geq \xi $. Then in $\left( a,b\right) $, the extremum is first
assumed at a point $w$ such that $x^{\prime }(w)=0$. By Corollaries {\ref%
{Corollary5.40}, \ref{Corollary5.41}, and \eqref {5.44}, w}e have 
\begin{eqnarray*}
w &\leq &b-\vartheta _{\tau _{m}} \\
w &\geq &a+\Psi \left( \frac{\max_{t\in \lbrack \xi -\tau _{m}-\vartheta
_{\tau _{m}},\xi ]}|x(t)|}{\max_{t\in (a,b)}|x(t)|},\tau _{m}\right) ,
\end{eqnarray*}%
where $\Psi (\rho ,\Delta )$ is defined in\eqref {7.11}.

Under \eqref {10.21}, considering Lemma {\ref{Lemma2.47}, there exist }$%
\varepsilon >0$ and $\rho \in (r_{\Delta }(\varepsilon ),1)$ such that 
\begin{equation*}
\Psi (\frac{1}{\rho },\tau _{m})+\vartheta _{\tau _{m}}-\varepsilon >(b-a).
\end{equation*}%
We will similarly show that $|x(u)|\leq \rho \max_{t\in \lbrack \xi -\tau
_{m}-\vartheta _{\tau _{m}},\xi ]}|x(t)|,u\geq \xi $. Assuming the contrary,
by considering the point 
\begin{equation*}
\psi :=\inf \{u>\xi :|x(u)|>\rho \max_{t\in \lbrack \xi -\tau _{m}-\vartheta
_{\tau _{m}},\xi ]}|x(t)|\},
\end{equation*}%
it is easy to see that this point is in a semicycle $(a,b)$ such that 
\begin{equation}
\max_{t\in (a,b)}|x(t)|>\rho \max_{t\in \lbrack \xi -\tau _{m}-\vartheta
_{\tau _{m}},a]}|x(t)|=\rho \max_{t\in \lbrack \xi -\tau _{m}-\vartheta
_{\tau _{m}},\xi ]}|x(t)|,  \label{11.49}
\end{equation}%
where $a\geq \xi $. Then in $\left( a,b\right) $, the extremum is first
assumed at a point $w$ such that $x^{\prime }(w)=0$. By Corollaries {\ref%
{Corollary5.40}, \ref{Corollary5.41}, and \eqref {11.49}, w}e have 
\begin{eqnarray*}
w &\leq &b-\left( \vartheta _{\tau _{m}}-\varepsilon \right) \\
w &\geq &a+\Psi \left( \frac{\max_{t\in \lbrack \xi -\tau _{m}-\vartheta
_{\tau _{m}},\xi ]}|x(t)|}{\max_{t\in (a,b)}|x(t)|},\tau _{m}\right) ,
\end{eqnarray*}%
a contradiction. The required result follows by considering any sequence of
zeros $\zeta _{n}$ with $\zeta _{1}:=\xi $ and $\zeta _{n+1}>\zeta _{n}+\tau
_{m}+\vartheta _{\tau _{m}}$. By induction, one obtains 
\begin{equation*}
|x(u)|\leq \rho ^{n}\max_{t\in \lbrack \xi -\tau _{m}-\vartheta _{\tau
_{m}},\xi ]}|x(t)|,u\geq \zeta _{n}.
\end{equation*}
\end{proof}

\section{Classification of solutions for\ negative coefficient}

We now investigate the classification of solutions of \eqref {2.41} with
nonpositive coefficient $p(t)\leq 0,t\in 
\mathbb{R}
$. The following was proven by Kamenskii {\cite{kamen} }for continuous
parameters and the same proof holds in the measurable case.

\begin{lemma}[\protect\cite{kamen}]
Let $x$ be a positive solution of \eqref {2.41} with nonpositive coefficient 
$p(t)\leq 0,t\in 
\mathbb{R}
$, and assume 
\begin{equation*}
\int_{0}^{\infty }|p(t)|dt=\infty .
\end{equation*}%
Then $x(t)$ has one of the following two asymptotic behaviors:
\end{lemma}

$i$)%
\begin{equation*}
\lim_{t\rightarrow \infty }x(t)=\lim_{t\rightarrow \infty }x^{\prime
}(t)=+\infty
\end{equation*}%
$ii$) 
\begin{equation*}
\lim_{t\rightarrow \infty }x(t)=0,\text{ and }x^{\prime }(t)\leq
0,\lim_{t\rightarrow \infty }x^{\prime }(t)=0.
\end{equation*}

The existence and uniqueness of such behaviors, with given initial
conditions, was investigated in {\cite{ladde, shmul, skub}. }Concerning the
possibility of decreasing positive solutions, we have the following results,
which we present within the framework of Azbelev \cite{Azb1971}, allowing
for zero initial function and discontinuity of the solution or its
derivative at the initial point. We recall the definition of the Wronskian
of the fundamental system 
\begin{equation*}
W(t):=\left\vert 
\begin{array}{cc}
z(t) & y(t) \\ 
z^{\prime }(t) & y^{\prime }(t)%
\end{array}%
\right\vert ,t\geq 0,
\end{equation*}%
where $z,y$ are the solutions of \eqref {2.41} on $[0,+\infty )$, satisfying 
$z(t)=y(t)=0,t<0$ and $z(0)=1,z^{\prime }(0)=0,$ and $y^{\prime
}(0)=1,y(0)=0.$

\begin{lemma}[{\protect\cite{gustafson}, {\protect\cite[Theorem 4.3.1]{ladde}}%
}]
Assume that $p(t)\leq 0,t\in 
\mathbb{R}
$,\ the function $t\mapsto (t-\tau (t))$ is nondecreasing, and that%
\begin{equation}
\lim \sup_{t\rightarrow \infty }\int_{t-\tau (t)}^{t}(s-t+\tau
(t))|p(s)|ds>1.  \label{3.07}
\end{equation}%
Then all nonoscillatory solutions of \eqref {2.41} are unbounded.
\end{lemma}

The following result was first proven in {\cite[p. 74]{labophd}.} For
extensions to higher-order equations, we refer the reader to {\cite{dom14,
labo, labophd, labo2}.}

\begin{lemma}[{{\ \protect\cite[Theorem 17.14]{agar}, \protect\cite[Theorem
10.2]{BDK}, \protect\cite[p. 74]{labophd}}}]
\label{Lemma12}Assume that $p(t)\leq 0,t\in 
\mathbb{R}
$. The existence of positive nonincreasing solutions of the inequality 
\begin{equation}
x^{\prime \prime }(t)+p(t)x(t-\tau (t))\geq 0,t\in \lbrack 0,+\infty ),
\label{5.30}
\end{equation}%
is equivalent to existence of positive nonincreasing solutions of the
equation%
\begin{equation}
x^{\prime \prime }(t)+p(t)x(t-\tau (t))=0,t\in \lbrack 0,+\infty ),
\label{5.31}
\end{equation}%
and to the nonvanishing of the Wronskian of the fundamental system on $%
[0,+\infty )$.
\end{lemma}

\begin{remark}
Lemma {\ref{Lemma12}} {holds in the general case of not necessarily bounded
delay and integrable coefficient.}
\end{remark}

\begin{lemma}[{{\ \protect\cite[Theorem 17.14]{agar}, \protect\cite[Theorem
10.2]{BDK}, \protect\cite{labo}, \protect\cite[p. 74]{labophd}}}]
Assume that 
\begin{equation*}
\tau _{m}\sqrt{\limfunc{esssup}_{t\in 
\mathbb{R}
}|p(t)|}\leq \frac{2}{e}.
\end{equation*}%
Then the Wronskian of the fundamental system of \eqref {2.41} has no zeros
on $[0,+\infty )$.
\end{lemma}

\begin{remark}
\label{Remark5.53}In virtue of Lemma {\ref{Lemma2.47} and} the monotonicity
of $\Psi (1,\Delta )$, defined in \eqref {7.11}, setting $\gamma $ as the
unique solution of 
\begin{equation}
\Psi (1,\gamma )=\gamma \in \lbrack \sqrt{2},\frac{\pi }{2}],  \label{12.53}
\end{equation}%
we have%
\begin{equation}
\Psi (1,\Delta )>\Delta ,\text{ if }\Delta <\gamma .  \label{5.56}
\end{equation}
\end{remark}

\bigskip For oscillatory solutions of \eqref {2.41} with nonpositive
coefficient $p(t)\leq 0,t\in 
\mathbb{R}
,$ we have the following result.

\begin{theorem}
\label{Theorem3.14}Consider an oscillatory solution $x$\ of \eqref {2.41},
with $p(t)\leq 0,t\in 
\mathbb{R}
$ and assume that 
\begin{equation}
\tau _{m}\sqrt{\limfunc{esssup}_{t\in 
\mathbb{R}
}|p(t)|}\leq \gamma ,  \label{10.34}
\end{equation}%
where $\gamma $ is defined in \eqref {12.53}. Then $x(t)$ is bounded. If,
further, the strict inequality holds in \eqref{10.34}, then $x(t)$ tends to
zero at infinity.
\end{theorem}

\begin{proof}
By Lemma {\ref{Lemma3.32}, it suffices to consider the case }$\limfunc{esssup%
}_{t\in 
\mathbb{R}
}|p(t)|=1$ (the case of the ODE, $\limfunc{esssup}_{t\in 
\mathbb{R}
}|p(t)|=0$, is trivial).{\ }Let us consider a point $\xi $ such that $x(\xi
)=0$. We shall show that%
\begin{equation}
|x(u)|\leq \max_{t\in \lbrack \xi -\tau _{m}-\vartheta _{\tau _{m}},\xi
]}|x(t)|,u\geq \xi ,  \label{9.47}
\end{equation}%
where $\vartheta _{(\cdot )\text{ }}$is as in Lemma {\ref{Lemma6.42}. }We
may assume $\max_{t\in \lbrack \xi -\tau _{m}-\vartheta _{\tau _{m}},\xi
]}|x(t)|>0$ without loss of generality. Assuming the contrary, that %
\eqref{9.47} does not hold, by considering the point 
\begin{equation*}
\psi :=\inf \{u>\xi :|x(u)|>\max_{t\in \lbrack \xi -\tau _{m}-\vartheta
_{\tau _{m}},\xi ]}|x(t)|\},
\end{equation*}%
it is easy to see that this point is in a semicycle $(a,b)$ such that%
\begin{equation}
\max_{t\in (a,b)}|x(t)|>\max_{t\in \lbrack \xi -\tau _{m}-\vartheta _{\tau
_{m}},a]}|x(t)|=\max_{t\in \lbrack \xi -\tau _{m}-\vartheta _{\tau _{m}},\xi
]}|x(t)|,  \label{5.55}
\end{equation}%
where $a\geq \xi $. We may assume that the extremum of this semicycle is
positive, without loss of generality. Then in $\left( a,b\right) $, the
maximum is first assumed at a point $w$ such that $x^{\prime }(w)=0.$ Using %
\eqref {2.41}, and $p(t)\leq 0,t\in 
\mathbb{R}
,$ there exists a set of positive measure consisting of points $v>w$ such
that $x^{\prime \prime }(v)<0,$ and also $\tau (v)<a$. By Corollary{\ \ref%
{Corollary5.41}, and \eqref{5.55}, we have }for such points $v$ 
\begin{eqnarray*}
v &>&w\geq a+\Psi \left( \frac{\max_{t\in \lbrack \xi -\tau _{m}-\vartheta
_{\tau _{m}},\xi ]}|x(t)|}{\max_{t\in (a,b)}|x(t)|},\tau _{m}\right) \\
a &>&\tau (v),
\end{eqnarray*}%
where $\Psi (\rho ,\Delta )$ is defined in \eqref {7.11}.

Under the strict inequality in \eqref {10.34} we have by {\eqref{5.56}, and
Lemma \ref{Lemma2.47}, that} $\Psi (\rho ,\tau _{m})\geq \tau _{m}$ for a $%
\rho >1$.

We will similarly show that 
\begin{equation*}
|x(u)|\leq \frac{1}{\rho }\max_{t\in \lbrack \xi -\tau _{m}-\vartheta _{\tau
_{m}},\xi ]}|x(t)|,u\geq \xi .
\end{equation*}%
Assuming the contrary, by considering the point 
\begin{equation*}
\psi :=\inf \{u>\xi :|x(u)|>\frac{1}{\rho }\max_{t\in \lbrack \xi -\tau
_{m}-\vartheta _{\tau _{m}},\xi ]}|x(t)|\},
\end{equation*}%
it is easy to see that this point is in a semicycle $(a,b)$ such that%
\begin{equation}
\max_{t\in (a,b)}|x(t)|>\frac{1}{\rho }\max_{t\in \lbrack \xi -\tau
_{m}-\vartheta _{\tau _{m}},a]}|x(t)|=\frac{1}{\rho }\max_{t\in \lbrack \xi
-\tau _{m}-\vartheta _{\tau _{m}},\xi ]}|x(t)|,  \label{5.43}
\end{equation}%
where $a\geq \xi $. We may assume that the extremum of this semicycle is
positive, without loss of generality. Then in $\left( a,b\right) $, the
maximum is first assumed at a point $w$ such that $x^{\prime }(w)=0.$ Using %
\eqref {2.41}, and $p(t)\leq 0,t\in 
\mathbb{R}
,$ there exists a set of positive measure consisting of points $v>w$ such
that $x^{\prime \prime }(v)<0,$ and also $\tau (v)<a$. By Corollary{\ \ref%
{Corollary5.41}, and \eqref{5.43}, we have }for such points $v$ 
\begin{eqnarray*}
v &>&w\geq a+\Psi \left( \frac{\max_{t\in \lbrack \xi -\tau _{m}-\vartheta
_{\tau _{m}},\xi ]}|x(t)|}{\max_{t\in (a,b)}|x(t)|},\tau _{m}\right) \\
a &>&\tau (v).
\end{eqnarray*}%
The required result follows by considering any sequence of zeros $\zeta _{n}$
with $\zeta _{1}:=\xi $ and $\zeta _{n+1}>\zeta _{n}+\tau _{m}+\vartheta
_{\tau _{m}}$. By induction, one obtains 
\begin{equation*}
|x(u)|\leq \left( \frac{1}{\rho }\right) ^{n}\max_{t\in \lbrack \xi -\tau
_{m}-\vartheta _{\tau _{m}},\xi ]}|x(t)|,u\geq \zeta _{n}.
\end{equation*}
\end{proof}

\begin{corollary}
Assume $p(t)\leq 0,t\in 
\mathbb{R}
$,$\int_{0}^{\infty }|p(t)|dt=\infty $ and $\tau _{m}\sqrt{\limfunc{esssup}%
_{t\in 
\mathbb{R}
}|p(t)|}<\gamma ,$ where $\gamma $ is defined in \eqref {12.53}. Then any
given solution $x$ of \eqref{2.41} has exactly one of the following
asymptotic behaviors:
\end{corollary}

$i$)%
\begin{equation*}
\text{ }x\text{ is oscillatory and }\lim_{t\rightarrow \infty }x(t)=0\text{ }
\end{equation*}%
$ii$) 
\begin{equation*}
\text{ }x\text{ is nonoscillatory, tends to zero monotonically,}%
\lim_{t\rightarrow \infty }x^{\prime }(t)=0\text{ }
\end{equation*}%
$iii$) 
\begin{equation*}
\lim_{t\rightarrow \infty }x(t)=\lim_{t\rightarrow \infty }x^{\prime
}(t)=+\infty
\end{equation*}%
$iv$) 
\begin{equation*}
\lim_{t\rightarrow \infty }x(t)=\lim_{t\rightarrow \infty }x^{\prime
}(t)=-\infty .
\end{equation*}

If furthermore, $\tau _{m}\sqrt{\limfunc{esssup}_{t\in 
\mathbb{R}
}|p(t)|}\leq \frac{2}{e}$, there must exist solutions exhibiting $ii$).

\section{Discussion and open problems}

The following example illustrates the import and information of Theorem {\ref%
{Theorem10.41}. }

\begin{example}
Consider the equation%
\begin{equation}
x^{\prime \prime }(t)+x(t-4)=0.  \label{7.43}
\end{equation}%
Here $\tau _{m}=4>2\sqrt{2}$ $\approx 2.83$ and Theorem {\ref{Theorem10.41}
implies that unbounded solutions of }\eqref {7.43} have semicycles of length
greater than $2\sqrt{2}$, whereas oscillatory solutions with semicycles of
smaller length tend to zero. It is known \cite[p. 135]{Hale} that Equation %
\eqref{7.43} possesses a spectral decomposition, and its real eigensolutions
are of the form $x(t)=\func{Re}\left( \exp (\lambda t)\right) $, or $x(t)=%
\func{Im}\left( \exp (\lambda t)\right) $, where $\lambda $ is a solution of 
\begin{equation*}
\lambda ^{2}+\exp (-4\lambda )=0.
\end{equation*}%
Using the Lambert $W-$function (on which see \cite{lambert}), we may express
the eigenvalues $\lambda $ in the form%
\begin{equation*}
\lambda =\frac{1}{2}W_{n}(\pm 2i),n\in 
\mathbb{Z}
.
\end{equation*}%
Let us consider the eigenvalues 
\begin{eqnarray*}
\frac{1}{2}W_{0}(\pm 2i) &\approx &0.34\pm 0.37i, \\
\frac{1}{2}W_{1}(+2i) &\approx &-0.57+3.05i \\
\frac{1}{2}W_{1}(-2i) &\approx &-0.21+1.50i \\
\frac{1}{2}W_{2}(+2i) &\approx &-0.92+6.21i \\
\frac{1}{2}W_{2}(-2i) &\approx &-0.77+4.63i
\end{eqnarray*}%
We notice that, as a necessary consequence of Theorem {\ref{Theorem10.41},}
the unbounded solutions corresponding to eigenvalues with positive real part
have semicycles of length 
\begin{equation*}
\frac{\pi }{|\func{Im}\frac{1}{2}W_{0}(\pm 2i)|}\approx 8.45>2\sqrt{2}%
\approx 2.83,
\end{equation*}%
while the solutions with semicycles of length less than $2\sqrt{2}$ 
\begin{eqnarray*}
\frac{\pi }{|\func{Im}\frac{1}{2}W_{1}(+2i)|} &\approx &1.03<2\sqrt{2}%
\approx 2.83 \\
\frac{\pi }{|\func{Im}\frac{1}{2}W_{1}(-2i)|} &\approx &2.09<2\sqrt{2}%
\approx 2.83 \\
\frac{\pi }{|\func{Im}\frac{1}{2}W_{2}(+2i)|} &\approx &0.51<2\sqrt{2}%
\approx 2.83 \\
\frac{\pi }{|\func{Im}\frac{1}{2}W_{2}(-2i)|} &\approx &0.68<2\sqrt{2}%
\approx 2.83.
\end{eqnarray*}%
correspond to eigenvalues with negative real part, thus tend to zero.
\end{example}

We now present examples illustrating the sharpness of the bounds obtained in
Theorem {\ref{Theorem10.41} in the cases }$\tau _{m}=0,\tau _{m}\geq 2\sqrt{2%
}$.

\begin{example}
\label{ex2} Fix an arbitrarily small $\varepsilon \geq 0$. Then by the
inequality $\tan (x)>\tanh (x),x\in (0,\frac{\pi }{2}),$ we have $%
\varepsilon -\arctan \left( \frac{\sinh (\varepsilon )}{\cosh \left(
\varepsilon \right) }\right) \geq 0$, and we set $A:=\pi +\varepsilon
-\arctan \frac{\sinh (\varepsilon )}{\cosh \left( \varepsilon \right) }\geq
\pi .$ Consider the solution of 
\begin{eqnarray*}
x^{\prime \prime }(t)+p(t)x(t)=0 && \\
p(t):=\left\{ 
\begin{array}{cc}
-1, & t\in [ nA,nA+\varepsilon ] \\ 
+1, & t\in [ nA+\varepsilon ,(n+1)A]%
\end{array}%
\right. && \\
x(0)=0,x^{\prime }(0)=1 &&
\end{eqnarray*}%
where $n=0,1,2,...$ Obviously, 
\begin{equation*}
x(t)=\left\{ 
\begin{array}{l}
\left( -\sqrt{\left( \sinh (\varepsilon )\right) ^{2}+\left( \cosh
(\varepsilon )\right) ^{2}}\right) ^{n}\sinh \left( t-nA\right) , \\ 
t\in [ nA,nA+\varepsilon ], \\ 
\left( -1\right) ^{n}\left( \sqrt{\left( \sinh (\varepsilon )\right)
^{2}+\left( \cosh (\varepsilon )\right) ^{2}}\right) ^{n+1} \!\!\! \!\! \!
\sin \left[ t-nA-\varepsilon +\arctan \left( \frac{\sinh (\varepsilon )}{%
\cosh \left( \varepsilon \right) }\right) \right] , \\ 
t \in [ nA+\varepsilon ,(n+1)A]%
\end{array}%
\right.
\end{equation*}%
and $x(t)$ has semicycles of length $A,$ does not tend to zero when $A=\pi $%
, is unbounded when $A>\pi $.
\end{example}

\begin{example}
\label{ex3} Fix $B=2\sqrt{2}+2\varepsilon ,$ where $\varepsilon \geq 0$ may
be arbitrarily small, and consider the solution of 
\begin{eqnarray*}
y^{\prime \prime }(t)+p(t)y(t-\tau (t))=0 && \\
p(t):=\left\{ 
\begin{array}{ll}
-1, & t\in \lbrack 2nB,2nB+\sqrt{2}+\varepsilon ] \\ 
+1, & t\in \lbrack 2nB+\sqrt{2}+\varepsilon ,(2n+1)B]%
\end{array}%
\right. && \\
\tau (t):=\left\{ 
\begin{array}{cc}
t-(2nB-\sqrt{2}), & t\in \lbrack 2nB,2nB+\sqrt{2}] \\ 
t-(2nB+\sqrt{2}), & t\in \lbrack 2nB+\sqrt{2},2nB+\sqrt{2}+2\varepsilon ] \\ 
t-(2nB+\sqrt{2}+2\varepsilon ), & t\in \lbrack 2nB+\sqrt{2}+2\varepsilon
,(2n+1)B]%
\end{array}%
\right. && \\
y(-\sqrt{2})=-1,y(0)=0,y^{\prime }(0)=\sqrt{2}. &&
\end{eqnarray*}%
where $n=0,1,2,...$ Obviously we have%
\begin{equation*}
y(t)=\left\{ 
\begin{array}{l}
(-1)^{n}\left[ 1+\varepsilon ^{2}\right] ^{n}\left( 1-\frac{1}{2}\left( 
\sqrt{2}-\left( t-2nB\right) \right) ^{2}\right) , \\ 
t\in \lbrack 2nB,2nB+\sqrt{2}], \\ 
(-1)^{n}\left[ 1+\varepsilon ^{2}\right] ^{n}\left( 1+\frac{1}{2}\left(
t-2nB-\sqrt{2}\right) ^{2}\right) , \\ 
t\in \lbrack 2nB+\sqrt{2},2nB+\sqrt{2}+\varepsilon ], \\ 
(-1)^{n}\left[ 1+\varepsilon ^{2}\right] ^{n}\left[ 1+\frac{1}{2}\varepsilon
^{2}-\frac{1}{2}\left( t-\left[ 2nB+\sqrt{2}+\varepsilon \right] \right)
^{2}+\varepsilon \left( t-\left[ 2nB+\sqrt{2}+\varepsilon \right] \right) %
\right] , \\ 
t\in \lbrack 2nB+\sqrt{2}+\varepsilon ,2nB+\sqrt{2}+2\varepsilon ], \\ 
(-1)^{n}\left[ 1+\varepsilon ^{2}\right] ^{n+1}\left( 1-\frac{1}{2}\left( t-%
\left[ 2nB+\sqrt{2}+2\varepsilon \right] \right) ^{2}\right) , \\ 
t\in \lbrack 2nB+\sqrt{2}+2\varepsilon ,(2n+1)B].%
\end{array}%
\right.
\end{equation*}%
Hence, $y(t)$ has semicycles of length $B,$ does not tend to zero whenever $%
B=2\sqrt{2}$ and is unbounded when $B>2\sqrt{2}$. 
\end{example}

However, straightforward calculations show that similar periodic examples
are not generally possible in the case $\tau _{m}\in (0,2\sqrt{2})$, as
these would violate the continuity of the first derivative at the zeros. We
do note that the bounds on the semicycles variate within a relatively small
interval $\left[ 2\sqrt{2}\approx 2.83,\pi \approx 3.14\right] .$ Hence the
present results are rather close to the optimal case, yet deeper
investigation of the first derivative at the zeros is required in order to
obtain better bounds. Presently, it is unclear how the methods of this paper
can be extended so as to include this second dimension without regularity
restrictions such as slow oscillation or the separation of the zeros.
Perhaps the techniques of \cite{eliason} may prove useful in this regard.

Ultimately, one also desires upper bounds on the semicycles, as well as
estimates on semicycle length depending on the initial function (see partial
results in \cite[Chapter 12, Chapter 14]{BDK}, \cite{Domshlak}, {\cite%
{Myshkis51}}). {Thus, based on the initial data, one could determine the
semicycles, whence one would deduce the asymptotic behavior of the solution.}

Examples~\ref{ex2} and \ref{ex3} have sign-changing coefficient, and we do
not presume our results are sharp in the case where the coefficient $p$ has
stable sign. It would be interesting to explore further results in the vein
of Theorem {\ref{Theorem3.14}, \cite[p. 39]{Myshkis51}, which directly take
into account the sign of the coefficient and its influence on the semicycles.%
}

\begin{problem}
\bigskip Find the optimal constant $\varkappa >0$ such that under $\tau
_{m}<\varkappa $ and $-1\leq p(t)\leq 0,t\in 
\mathbb{R}
$, all oscillatory solutions of \eqref {2.41} tend to zero at infinity.
\end{problem}

Remark {\ref{Remark5.53}, }Theorem {\ref{Theorem3.14}, imply that }$%
\varkappa \geq \sqrt{2}$. The following Example shows that $\varkappa \leq
\pi $.

\begin{example}
The periodic function $\sin (t)$ solves 
\begin{equation*}
y^{\prime \prime }(t)=y(t-\pi ),t\in 
\mathbb{R}%
\end{equation*}%
with constant delay $\tau _{m}=\pi .$
\end{example}

All the known examples of solutions on the boundary between solutions that
tend to zero and unbounded solutions are periodic, for both the second-order
and the first-order equation. Presuming the optimal constant{\ corresponds
to periodic oscillatory solutions of \eqref {2.41}, one should have }$%
\varkappa \approx \pi $, as in the above Example{. }As one cannot assume
that an arbitrary oscillatory solution is periodic in the investigation of
this problem, it remains open.

\section{Acknowledgments}

The first author was supported by the NSERC Grant RGPIN-2020-03934. The third author acknowledges the support of Ariel University.

\end{document}